\documentclass{tac}
\usepackage[all,ps,arc]{xy}  
\usepackage{amsmath,amssymb,latexsym}
\usepackage{pdfsync}
\usepackage{makeidx}
\usepackage[english]{babel}
\usepackage[applemac]{inputenc}
\usepackage{graphicx}
\usepackage[x11names]{xcolor}

\newtheorem{Theorem}{Theorem}[section]

\newtheorem{Proposition}[Theorem]{Proposition}

\newtheorem{Text}[Theorem]{}
\newcommand{\bA}{{\mathbb A}}
\newcommand{\bB}{{\mathbb B}}
\newcommand{\bC}{{\mathbb C}}
\newcommand{\bD}{{\mathbb D}}
\newcommand{\bE}{{\mathbb E}}
\newcommand{\bF}{{\mathbb F}}
\newcommand{\bG}{{\mathbb G}}
\newcommand{\bH}{{\mathbb H}}
\newcommand{\bK}{{\mathbb K}}
\newcommand{\bP}{{\mathbb P}}
\newcommand{\bV}{{\mathbb V}}
\newcommand{\bX}{{\mathbb X}}
\newcommand{\bY}{{\mathbb Y}}

\newcommand{\cA}{{\mathcal A}}

\newcommand{\cC}{{\mathcal C}}
\newcommand{\cF}{{\mathcal F}}

\newcommand\Grp{\mathbf{Grp}}

\newcommand\Grpd{\mathbf{Grpd}}
\newcommand\Fract{\mathbf{Fract}}
\newcommand\Funct{\mathbf{Funct}}

\newcommand\id{\mathrm{id}}

\newcommand{\cd}{\xymatrix}

\title{Bipullbacks of fractions and the snail lemma}
\author{P.-A. Jacqmin, S. Mantovani, G. Metere, E.M. Vitale}
\amsclass{18B40, 18D05, 18E35, 18G50}
\keywords{bicategory of fractions, bipullback, internal groupoid, snail lemma}
\address
{Institut de recherche en math\'ematique et physique, Universit\'e catholique de Louvain\\
Chemin du Cyclotron 2, B 1348 Louvain-la-Neuve, Belgique.\\
Dipartimento di matematica, Universit\`a degli studi di Milano\\
Via C. Saldini 50, 20133 Milano, Italia.\\
Dipartimento di matematica e informatica, Universit\`a degli studi di Palermo\\
Via Archirafi 34, 90123 Palermo, Italia.}
\copyrightyear{2017}
\thanks{Financial support from FNRS grant 1.A741.14 is gratefully acknowledged by the first author. The second and the third author acknowledge the financial support of the I.N.D.A.M.\ Gruppo Nazionale per le Strutture Algebriche, Geometriche e le loro Applicazioni.}
\eaddress{pierre-alain.jacqmin@uclouvain.be,\\ sandra.mantovani@unimi.it,\\ giuseppe.metere@unipa.it,\\ enrico.vitale@uclouvain.be}

\begin{document}

\maketitle


\begin{abstract}
We establish conditions giving the existence of bipullbacks in bicategories of fractions.
We apply our results to construct a $\pi_0$-$\pi_1$ exact sequence associated with 
a fractor between groupoids internal to a pointed exact category.
\end{abstract}

\tableofcontents

\section{Introduction}

Let $\cA$ be a pointed exact (in the sense of~\cite{Barr}) category,
and let $F \colon \bA \to \bB$ be a functor between internal groupoids in $\cA$. 
In~\cite{SnailMMVShort}, we have constructed a $\pi_0$-$\pi_1$ exact sequence
$$\pi_1(\bK(F)) \to \pi_1(\bA) \to \pi_1(\bB) \to \pi_0(\bK(F)) \to \pi_0(\bA) \to \pi_0(\bB)$$
where $\pi_1(\bA)$ is the internal group of automorphisms on the base point of $\bA$, $\pi_0(\bA)$ is 
the object of connected components of $\bA$, and $\bK(F)$ is the bikernel of $F$. This sequence 
subsumes several relevant special cases: if $\cA$ is the category of pointed sets, we get the Gabriel-Zisman 
exact sequence of~\cite{GZ} and, if $F$ is a fibration, the Brown exact sequence of~\cite{RB}; if $\cA$ 
is semi-abelian or abelian, we get the exact sequence of the classical snake lemma when $F$ is a fibration, 
and the exact sequence of the more general snail lemma if $F$ is an arbitrary functor (see~\cite{DB, SnailZJEV, SnailEV}). 

A special case of particular interest is when $\cA$ is the category of groups. In this case, the $\pi_0$-$\pi_1$
sequence already appears in~\cite{DKV}; in fact, the sequence in~\cite{DKV} is even more general,
because it is obtained starting from a monoidal functor between groupoids in groups, and not necessarily from an
internal functor. The precise relation between internal functors and monoidal functors between internal 
groupoids in groups has been established in~\cite{FractEV}: monoidal functors are precisely fractions of internal functors 
with respect to weak equivalences (in the sense of~\cite{BP}). When $\cA$ is an arbitrary exact category, the bicategory 
of fractions of $\Grpd(\cA)$ with respect to weak equivalences has been described in~\cite{MMV} using fractors (fractors
are a particular kind of profunctors). The aim of this note is to complete the result of~\cite{SnailMMVShort}, showing
that the $\pi_0$-$\pi_1$ exact sequence can be constructed starting from any fractor $F \colon \bA \looparrowright \bB$
between internal groupoids in $\cA$ pointed exact. 

Since the 2-functors
$$\pi_0 \colon \Grpd(\cA) \to \cA \quad \text{and} \quad \pi_1 \colon \Grpd(\cA) \to \Grp(\cA)$$
send weak equivalences onto isomorphisms (see Lemma~4.5 in~\cite{SnailMMVShort}), they can be extended to the 
bicategory of fractions. Therefore, what remains to be done is to construct bipullbacks (and, in particular, bikernels) 
in the bicategory of fractions using bipullbacks in the 2-category of internal functors.
More generally, we prove that, if $\mathbf B$ is a bicategory with bipullbacks and $\Sigma$ has a right calculus of fractions in the sense of \cite{Pronk},  the bicategory of fractions
$$P_{\Sigma} \colon \mathbf B \to \mathbf B[\Sigma^{-1}]$$
preserves bipullbacks and $\mathbf B[\Sigma^{-1}]$ has them.
Since this is the case  for  the bicategory $\Grpd(\cA)$ with $\Sigma$ the
class of weak equivalences, we can extend the main result of~\cite{SnailMMVShort} getting a $\pi_0$-$\pi_1$
exact sequence from any internal fractor.

\hfill

In this paper, the composition of two arrows
$$\cd{ \ar[r]^-{f} & \ar[r]^-{g} &}$$
will be denoted as $f \cdot g$, or simply by $fg$.

In order to shorten notation, we will use coherence theorems for bicategories as in~\cite{MP}. Therefore, coherence isomorphisms will not be written explicitly.

We will sometimes need to change our base universe to a bigger one in order to define properly some (bi)categories.
However, we will omit to say it when it has to be done.

\section{A reminder on (bi)categories of fractions}

\begin{Text}\label{TextCatFract}{\rm
Categories of fractions have been introduced by P.\ Gabriel and M.\ Zisman in~\cite{GZ} as a useful tool for
algebraic topology and homotopy theory. They encode the universal solution to the problem of converting 
an arrow into an isomorphism. More precisely, if $\cA$ is a category and $\Sigma$ is a class of arrows in $\cA$, 
the category of fractions of $\cA$ with respect to $\Sigma$ is a functor
$$P_{\Sigma} \colon \cA \to \cA[\Sigma^{-1}]$$
universal among all functors $\cF \colon \cA \to \cC$ such that $\cF(s)$ is an isomorphism for all $s \in \Sigma$. 
In other words, for every category $\cC$ the functor induced by composition with $P_{\Sigma}$ is an equivalence 
of categories
$$P_{\Sigma} \cdot - \;\colon \Funct(\cA[\Sigma^{-1}],\cC) \to 
\Funct_{\Sigma}(\cA,\cC)$$
where $\Funct_{\Sigma}(\cA,\cC)$ is the category of those functors $\cF \colon \cA \to \cC$ such that $\cF(s)$ 
is an isomorphism for all $s \in \Sigma$. 

Usually, the description of the category of fractions is quite complicated (see Chapter~5 in~\cite{Bor} for a detailed 
discussion), but if the class $\Sigma$ has a right calculus of fractions, then the description of $\cA[\Sigma^{-1}]$
becomes much more clear. To have a right calculus of fractions means that:
\begin{enumerate}
\item[CF1.] $\Sigma$ contains the identity arrows.
\item[CF2.] $\Sigma$ is closed under composition.
\item[CF3.] Given arrows $f$ and $s$ with same codomain, if $s \in \Sigma$ then there exist arrows $s'$ and $f'$ 
such that $s' \cdot f = f' \cdot s$ and $s' \in \Sigma$.
$$\xymatrix{\ar[r]^-{f'} \ar[d]_{s'} & \ar[d]^{s} \\ \ar[r]_-{f} & }$$
\item[CF4.] Given parallel arrows $f$ and $g$, if there exists an arrow $s \in \Sigma$ such that $f \cdot s = g \cdot s$,
then there exists an arrow $s' \in \Sigma$ such that $s' \cdot f = s' \cdot g$.
$$\xymatrix{\ar[r]^-{s'} & \ar@<0,5ex>[r]^-{f} \ar@<-0,5ex>[r]_-{g} & \ar[r]^-{s} & }$$
\end{enumerate}
}\end{Text}

\begin{Text}\label{TextCatFractDescr}{\rm
If the class $\Sigma$ has a right calculus of fractions, the category $\cA[\Sigma^{-1}]$ can be described as follows:
\begin{enumerate}
\item[-] The objects of $\cA[\Sigma^{-1}]$ are those of $\cA$.
\item[-] An arrow in $\cA[\Sigma^{-1}]$ from $A$ to $B$ is an equivalence class of spans $(s,f)$ with $s \in \Sigma$
$$\xymatrix{A & \ar[l]_s I \ar[r]^f & B}$$  
two spans $(s,f)$ and $(s',f')$ being equivalent if there exist arrows $x,x'$ in $\cA$ such that 
$x \cdot s = x' \cdot s' \in \Sigma$ and $x \cdot f = x' \cdot f'$.
$$\xymatrix{ & I \ar[ld]_-{s} \ar[rd]^-{f} \\ A & X \ar[u]_-{x} \ar[d]^-{x'} & B \\ & I'  \ar[lu]^-{s'} \ar[ru]_-{f'} }$$
\end{enumerate}
Moreover, if $\cA$ has pullbacks and $\Sigma$ has a right calculus of fractions,
then $\cA[\Sigma^{-1}]$ has pullbacks and $P_{\Sigma} \colon \cA \to \cA[\Sigma^{-1}]$ preserves them~\cite{Bor}.
}\end{Text}

\begin{Text}\label{TextBicatFract}{\rm
If $\mathbf B$ is a bicategory and $\Sigma$ is a class of arrows in $\mathbf B$, the bicategory of fractions
of $\mathbf B$ with respect to $\Sigma$ is a morphism of bicategories
$$P_{\Sigma} \colon \mathbf B \to \mathbf B[\Sigma^{-1}]$$
which encodes the universal solution to the problem of sending an arrow of $\Sigma$ into an equivalence.
This problem has been studied by D.\ Pronk in~\cite{Pronk}, providing a bicategorical version of right calculus of fractions.
\begin{enumerate}
\item[BF1.] $\Sigma$ contains the equivalences.
\item[BF2.] $\Sigma$ is closed under composition.
\item[BF3.] Given arrows $F$ and $S$ with same codomain, if $S \in \Sigma$ then there exist arrows 
$S' \in \Sigma$ and $F'$ and a 2-iso $S' \cdot F \cong F' \cdot S$.
$$\xymatrix{\ar[r]^-{F'} \ar[d]_{S'} & \ar[d]^{S} \\ \ar[r]_-{F} \ar@{}[ru]|{\cong} & }$$
\item[BF4.] For every 2-cell (resp.\ 2-iso) $\alpha \colon F \cdot W \Rightarrow G \cdot W$ with $W \in \Sigma$ 
there exist a $V \in \Sigma$ and a 2-cell (resp.\ 2-iso) $\beta \colon V \cdot F \Rightarrow V \cdot G$ such that 
$V \cdot \alpha = \beta \cdot W$ 
$$\xymatrix{\ar[r]^-{V} & \ar@<0,5ex>[r]^-{F} \ar@<-0,5ex>[r]_-{G} & \ar[r]^-{W} & }$$
and for any two such pairs $(V,\beta)$ and $(V', \beta')$, there exist 
$U, U'$ and a 2-iso $\varepsilon \colon U \cdot V \Rightarrow U' \cdot V'$
$$\xymatrix{&\ar[rd]^-{V} \ar@{}[dd]|(.35){}="a" \ar@{}[dd]|(.65){}="b" \ar@{=>}"a";"b"^(.3)*[@]{\sim}_-{\varepsilon} &&& \\ \ar[ru]^-{U} \ar[rd]_-{U'} && \ar@<0,5ex>[r]^-{F} \ar@<-0,5ex>[r]_-{G} & \ar[r]^-{W} & \\ & \ar[ru]_-{V'} &&&}$$
such that $U \cdot V \in \Sigma$ and the diagram
$$\xymatrix{U \cdot V \cdot F \ar@{=>}[rr]^{U \cdot \beta} \ar@{=>}[d]_{\varepsilon
\cdot F} & & U \cdot V \cdot G \ar@{=>}[d]^{\varepsilon \cdot G} \\
U' \cdot V' \cdot F \ar@{=>}[rr]_{U' \cdot \beta'} & & U' \cdot V' \cdot G}$$
commutes.
\item[BF5.] Given arrows $F,G$ and a 2-iso $F \cong G$, then $F \in \Sigma$ if and only if $G \in \Sigma$.
\end{enumerate}
}\end{Text}

\begin{Text}\label{TextBicatFractDescr}{\rm
If the class $\Sigma$ has a right calculus of fractions, the bicategory $\mathbf B[\Sigma^{-1}]$ can be described as follows:
\begin{enumerate}
\item[-] The objects of $\mathbf B[\Sigma^{-1}]$ are those of $\mathbf B$.
\item[-] 1-cells $\bA \rightarrow \bB$ in $\mathbf B[\Sigma^{-1}]$ are spans $(W,F)$ with $W \in \Sigma$.
$$\cd{\bA & \bC \ar[l]_-{W} \ar[r]^-{F} & \bB}$$
\item[-] 2-cells $(W,F) \Rightarrow (V,G)$ are equivalent classes of quadruples $(U_1,U_2,\alpha_1, \alpha_2)$ where $U_1 \cdot W \in \Sigma$,
$\alpha_1 \colon U_1 \cdot W \Rightarrow U_2 \cdot V$ is a 2-iso and $\alpha_2 \colon U_1 \cdot F \Rightarrow U_2 \cdot G$ is a 2-cell.
$$\cd{ & \bC \ar[rd]^-{W} \ar@/^/[rrrd]^-{F} \ar@{}[dd]|(.35){}="a" \ar@{}[dd]|(.65){}="b" \ar@{=>}"a";"b"^(.35)*[@]{\sim}_-{\alpha_1} && \ar@{}@<-9pt>[dd]|(.35){}="c" \ar@{}@<-9pt>[dd]|(.65){}="d" & \\ \bE \ar[ru]^-{U_1} \ar[rd]_{U_2} && \bA && \bB \\
& \bD \ar[ru]_{V} \ar@/_/[rrru]_{G} &&& \ar@{=>}"c";"d"^-{\alpha_2}}$$
Two quadruples $(U_1,U_2,\alpha_1, \alpha_2)$ and $(U_1',U_2',\alpha_1', \alpha_2')$ are equivalent when there exist
$$\cd{ &&& \bC \ar[rd]^-{W} \ar@/^/[rrrd]^-{F} \ar@{}[dd]|(.35){}="a" \ar@{}[dd]|(.65){}="b" \ar@{=>}"a";"b"^(.35)*[@]{\sim}_-{\alpha_1} && \ar@{}@<-9pt>[dd]|(.35){}="c" \ar@{}@<-9pt>[dd]|(.65){}="d" & \\
\bE' \ar@/^1pc/[rrru]^-{U_1'}="U1'" \ar@/_1pc/[rrrd]_-{U_2'}="U2'" & \bF \ar[r]^-{R_1} \ar[l]_-{R_2} & \bE \ar[ru]^-{U_1} \ar[rd]_{U_2} && \bA && \bB \\
&&& \bD \ar[ru]_{V} \ar@/_/[rrru]_{G} &&& 
\ar@{=>}"c";"d"^-{\alpha_2} \ar@{=>}"U1'";[llllu]^-{\gamma_1}_(.4)*[@]-{\sim} \ar@{=>}[llllu];"U2'"^-{\gamma_2}_(.3)*[@]-{\sim}}$$
$R_1 \colon \bF \rightarrow \bE$, $R_2 \colon \bF \rightarrow \bE'$ and two 2-isos
$\gamma_1 \colon R_2 \cdot U_1' \Rightarrow R_1 \cdot U_1$ and $\gamma_2 \colon R_1 \cdot U_2 \Rightarrow R_2 \cdot U_2'$ such that 
$R_1 \cdot U_1 \cdot W \in \Sigma$ and
the following diagrams commute.
$$\cd{R_2 \cdot U_1' \cdot W \ar@{=>}[r]^-{\gamma_1 \cdot W} \ar@{=>}[d]_-{R_2 \cdot \alpha_1'} & 
R_1 \cdot U_1 \cdot W \ar@{=>}[d]^-{R_1 \cdot \alpha_1} \\
R_2 \cdot U_2' \cdot V & R_1 \cdot U_2 \cdot V \ar@{=>}[l]^-{\gamma_2 \cdot V}}
\qquad \cd{R_2 \cdot U_1' \cdot F \ar@{=>}[r]^-{\gamma_1 \cdot F} \ar@{=>}[d]_-{R_2 \cdot \alpha_2'} & 
R_1 \cdot U_1 \cdot F \ar@{=>}[d]^-{R_1 \cdot \alpha_2} \\
R_2 \cdot U_2' \cdot G & R_1 \cdot U_2 \cdot G \ar@{=>}[l]^-{\gamma_2 \cdot G}}$$
\item[-] $1$-cells composition  is obtained by completing two consecutive spans with the span provided by BF3, and then forgetting the $2$-cell. Vertical composition of $2$-cells is obtained by pasting two vertically consecutive 2-cells with a square coming from BF3, horizontal composition is described by a vertical composition of two whiskerings. More details on compositions and identities can be found in \cite{Pronk}.

\end{enumerate}
The 2-isos in $\mathbf B[\Sigma^{-1}]$ are exactly the 2-cells which can be represented by a quadruple $(U_1,U_2,\alpha_1, \alpha_2)$
where $\alpha_2$ is a 2-iso of $\mathbf B$. Moreover, the universal morphism 
$$P_{\Sigma} \colon \mathbf B \rightarrow \mathbf B[\Sigma^{-1}]$$ 
is defined by: $P_{\Sigma}(\bA)=\bA$, $P_{\Sigma}(F)=(1_{\bA},F)$ and $P_{\Sigma}(\alpha)=[(1_{\bA},1_{\bA},1_{1_{\bA}},\alpha)]$. 
For more details, see~\cite{Pronk}.
}\end{Text}

\begin{Text}\label{TextBipb}{\rm
Recall that a bipullback of two arrows $F \colon \bA \to \bB$ and $G \colon \bC \to \bB$ in a bicategory $\mathbf B$
is a diagram of the form
$$\xymatrix{\bP \ar[r]^{G'} \ar[d]_{F'} & \bA \ar[d]^{F} \\
\bC \ar@{}[ru]|(.35){}="a" \ar@{}[ru]|(.65){}="z" \ar@{=>}"a";"z"_{\pi} \ar@{}@<-6pt>"a";"z"^*[@]{\sim} \ar[r]_{G} & \bB}$$
satisfying the following universal property: 
\begin{enumerate}
\item[BP1.] For any diagram of the form 
$$\xymatrix{\bX \ar[r]^{H} \ar[d]_{K} & \bA \ar[d]^{F} \\
\bC \ar@{}[ru]|(.35){}="a" \ar@{}[ru]|(.65){}="z" \ar@{=>}"a";"z"_{\mu} \ar@{}@<-6pt>"a";"z"^*[@]{\sim} \ar[r]_{G} & \bB}$$
there exist an arrow $T \colon \bX \to \bP$ and 2-isos $\gamma \colon T \cdot G' \Rightarrow H, \delta \colon T \cdot F' \Rightarrow K$ 
making commutative the following diagram.
$$\cd{T \cdot F' \cdot G \ar@{=>}[d]_-{T \cdot \pi} \ar@{=>}[r]^-{\delta \cdot G} & K \cdot G \ar@{=>}[d]^-{\mu} \\
T \cdot G' \cdot F \ar@{=>}[r]_-{\gamma \cdot F} & H \cdot F}$$
\item[BP2.] Given 1-cells $T, S \colon \bX \rightrightarrows \bP$ and 2-isos $\alpha \colon T \cdot F' \Rightarrow  S \cdot F'$ and
$\beta \colon T \cdot G' \Rightarrow S \cdot G'$, if
$$\xymatrix{T \cdot F' \cdot G \ar@{=>}[r]^-{\alpha \cdot G} \ar@{=>}[d]_{T \cdot \pi} & S \cdot F' \cdot G \ar@{=>}[d]^{S \cdot \pi} \\
T \cdot G' \cdot F \ar@{=>}[r]_-{\beta \cdot F} & S \cdot G' \cdot F}$$
commutes, then there exists a unique 2-cell $\varphi \colon T \Rightarrow S$ such that $\varphi \cdot F' = \alpha$ and $\varphi \cdot G' = \beta$.
(Observe that the 2-cell $\varphi$ of BP2 is actually a 2-iso.)
\end{enumerate}
}\end{Text}

\begin{Text}\label{TextBipbHpb}{\rm
In~\cite{SnailMMVShort}, we focused on strong h-pullbacks instead of bipullbacks. The universal property of strong h-pullbacks 
subsumes that of bipullbacks, but strong h-pullbacks are determined up to isomorphism, whereas bipullbacks are determined up to 
equivalence. This is why the notion of bipullback is the `correct' notion of 2-dimensional pullback in the context of bicategories of fractions.
Moreover, we will use the fact that pasting together bipullbacks we still get a bipullback (which is not the case if we work with strong h-pullbacks).
See~\cite{Ben, JMMVFibr} for more details on bipullbacks and strong h-pullbacks.
}\end{Text}

\section{Preservation of bipullbacks}

This section is devoted to the proof of the following proposition, that generalises the 1-dimensional case.

\begin{Proposition}\label{PropBPBFract1}
Let $\mathbf B$ be a bicategory with bipullbacks and $\Sigma$ a class of arrows in $\mathbf B$ having a right calculus of 
fractions. Then $\mathbf B[\Sigma^{-1}]$ has bipullbacks and the universal morphism $P_{\Sigma} \colon \mathbf B \rightarrow \mathbf B[\Sigma^{-1}]$ preserves them.
\end{Proposition}

\begin{proof}
Consider two arrows in $\mathbf B[\Sigma^{-1}]$
$$\xymatrix{ & \bF \ar[ld]_{S} \ar[rd]^{R} \\ \bA && \bB} \qquad \xymatrix{ & \bG \ar[ld]_{U} \ar[rd]^{T} \\ \bC && \bB}$$
with $S, U \in \Sigma,$ and the bipullback of the right legs in $\mathbf B$.
$$\xymatrix{\bP \ar[d]_{R'} \ar[r]^-{T'} & \bF \ar[d]^{R} \\
\bG \ar[r]_-{T} \ar@{}[ru]|(.35){}="a" \ar@{}[ru]|(.65){}="z" \ar@{=>}"a";"z"_{\pi} \ar@{}@<-6pt>"a";"z"^*[@]{\sim} & \bB}$$
If we can prove that 
$$\xymatrix{\bP \ar[d]_{(1,R')} \ar[rr]^-{(1,T')} && \bF \ar[d]^{(1,R)} \\
\bG \ar[rr]_-{(1,T)} \ar@{}[rru]|(.35){}="a" \ar@{}[rru]|(.65){}="z" \ar@{=>}"a";"z"_-{P_{\Sigma}(\pi)} \ar@{}@<-4pt>"a";"z"^*[@]{\sim} && \bB}$$
is still a bipullback in $B[\Sigma^{-1}]$, then the bipullback of $(S,R)$ and $(U,T)$ in $B[\Sigma^{-1}]$ is obtained pasting together the four bipullbacks below.
$$\xymatrix{\bP \ar[d]_{(1,R')} \ar[rr]^-{(1,T')} & & \bF \ar[d]^{(1,R)} \ar[rr]^-{(1,S)} & & \bA \ar[d]^{(S,R)} \\
\bG \ar[d]_{(1,U)} \ar[rr]_-{(1,T)} \ar@{}[rru]|(.35){}="a" \ar@{}[rru]|(.65){}="z" \ar@{=>}"a";"z"_-{P_{\Sigma}(\pi)} \ar@{}@<-4pt>"a";"z"^*[@]{\sim} & &
\bB \ar[d]^{1_{\bB}} \ar[rr]_-{1_{\bB}} \ar@{}[rru]|{\cong} & & \bB \ar[d]^{1_{\bB}} \\
\bC \ar[rr]_-{(U,T)} \ar@{}[rru]|{\cong} & & \bB \ar[rr]_-{1_{\bB}} \ar@{}[rru]|{\cong} & & \bB}$$

The fact that the lower-right square is a bipullback is obvious, and that the upper-right and lower-left are bipullbacks is easily proved. Indeed, for the upper-right square, just compose the obvious bipullback
$$
\xymatrix@C=11ex{
\bA\ar[r]^{1_{\bA}}\ar[d]_{(S,R)}\ar@{}[dr]|{\cong}
&\bA\ar[d]^{(S,R)}
\\
\bB\ar[r]_{1_{\bB}}
&\bB
}
$$ 
with the equivalence $(1,S)\colon \bF \to \bA$. The lower-left one is treated similarly.

Hence let us prove that the upper left square is a bipullback. 
We are going to prove separately the two parts of the universal property of the bipullback, that is, BP1 and BP2.

BP1. Consider a diagram in $\mathbf B[\Sigma^{-1}]$
$$\xymatrix{\bV \ar[d]_{(Y_1,Y_2)} \ar[rr]^-{(X_1,X_2)} && \bF \ar[d]^{(1,R)} \\
\bG \ar[rr]_-{(1,T)} \ar@{}[rru]|(.35){}="a" \ar@{}[rru]|(.65){}="z" \ar@{=>}"a";"z"_-{\mu} \ar@{}@<-4pt>"a";"z"^*[@]{\sim} && \bB}$$
where $\mu$ is represented by
$$\cd{ & \bY \ar[rd]^-{Y_1} \ar@/^/[rrrd]^-{Y_2 \cdot T}="F" \ar@{}[dd]|(.35){}="a" \ar@{}[dd]|(.65){}="b" \ar@{=>}"a";"b"^(.35)*[@]{\sim}_-{\mu_1} && \ar@{}@<-9pt>[dd]|(.35){}="c" \ar@{}@<-9pt>[dd]|(.65){}="d" & \\ \bE \ar[ru]^-{U_1} \ar[rd]_{U_2} 
&& \bV && \bB \\
& \bX \ar[ru]_{X_1} \ar@/_/[rrru]_{X_2 \cdot R}="G" &&& \ar@{=>}"c";"d"^(.35)*[@]{\sim}_-{\mu_2}}$$
We thus get the following diagram in $\mathbf B$
$$\cd{\bE \ar[r]^{U_2 \cdot X_2} \ar[d]_{U_1 \cdot Y_2} & \bF \ar[d]^{R} \\
\bG \ar@{}[ru]|(.35){}="a" \ar@{}[ru]|(.65){}="z" \ar@{=>}"a";"z"_{\mu_2} \ar@{}@<-6pt>"a";"z"^*[@]{\sim} \ar[r]_{T} & \bB}$$
and by the universal property of the bipullback BP1, an arrow $L \colon \bE \rightarrow \bP$ and two 2-isos 
$\gamma \colon L \cdot T' \Rightarrow U_2 \cdot X_2$ and
$\delta \colon L \cdot R' \Rightarrow U_1 \cdot Y_2$ which make the diagram
$$\cd{L \cdot R' \cdot T \ar@{=>}[d]_-{L \cdot \pi} \ar@{=>}[r]^-{\delta \cdot T} & U_1 \cdot Y_2 \cdot T \ar@{=>}[d]^-{\mu_2} \\
L \cdot T' \cdot R \ar@{=>}[r]_-{\gamma \cdot R} & U_2 \cdot X_2 \cdot R}$$
commutative. This gives us a 1-cell $(U_1 \cdot Y_1, L) \colon \bV \rightarrow \bP$ in $\mathbf B[\Sigma^{-1}]$ and two 2-isos 
$$[(1_{\bE},U_2,\mu_1,\gamma)] \colon (U_1 \cdot Y_1, L) \cdot (1_{\bP},T') \Rightarrow (X_1,X_2)$$ and
$$[(1_{\bE},U_1,1_{U_1 \cdot Y_1},\delta)] \colon (U_1 \cdot Y_1, L) \cdot (1_{\bP},R') \Rightarrow (Y_1,Y_2).$$
$$\cd{ & \bE \ar[rd]|-{U_1 \cdot Y_1} \ar@/^/[rrrd]^-{L \cdot T'}="F" \ar@{}[dd]|(.35){}="a" \ar@{}[dd]|(.65){}="b" \ar@{=>}"a";"b"^(.35)*[@]{\sim}_-{\mu_1} && \ar@{}@<-9pt>[dd]|(.35){}="c" \ar@{}@<-9pt>[dd]|(.65){}="d" & \\
 \bE \ar[ru]^-{1_{\bE}} \ar[rd]_{U_2} && \bV && \bF \\
& \bX \ar[ru]|-{X_1} \ar@/_/[rrru]_{X_2}="G" &&& \ar@{=>}"c";"d"^(.35)*[@]{\sim}_-{\gamma}} \qquad
\cd{ & \bE \ar[rd]|-{U_1 \cdot Y_1} \ar@/^/[rrrd]^-{L \cdot R'}="F" \ar@{}[dd]|(.35){}="a" \ar@{}[dd]|(.65){}="b" \ar@{=>}"a";"b"^(.35)*[@]{\sim}_-{1_{U_1 \cdot Y_1}} && \ar@{}@<-9pt>[dd]|(.35){}="c" \ar@{}@<-9pt>[dd]|(.65){}="d" & \\ \bE \ar[ru]^-{1_{\bE}} \ar[rd]_{U_1} && \bV && \bG \\
& \bY \ar[ru]|-{Y_1} \ar@/_/[rrru]_{Y_2}="G" &&& \ar@{=>}"c";"d"^(.35)*[@]{\sim}_-{\delta}}$$
The compatibility condition linking those two 2-isos can be deduced from the one linking $\gamma$ and $\delta$.

BP2. Now, suppose we have two arrows $\bV \rightrightarrows \bP$ in $\mathbf B[\Sigma^{-1}]$
$$\cd{ & \bH_i \ar[ld]_{W_i} \ar[rd]^{H_i} \\ \bV && \bP}$$
for $i \in \{1,2\}$ together with two 2-isos 
$$\alpha=[(U_1,U_2,\alpha_1,\alpha_2)] \colon (W_1,H_1) \cdot (1_{\bP},R') \Rightarrow (W_2,H_2) \cdot (1_{\bP},R')$$ and
$$\beta=[(V_1,V_2,\beta_1,\beta_2)] \colon (W_1,H_1) \cdot (1_{\bP},T') \Rightarrow (W_2,H_2) \cdot (1_{\bP},T').$$
$$\cd{ & \bH_1 \ar[rd]^-{W_1} \ar@/^/[rrrd]^-{H_1 \cdot R'}="F" \ar@{}[dd]|(.35){}="a" \ar@{}[dd]|(.65){}="b" \ar@{=>}"a";"b"^(.35)*[@]{\sim}_-{\alpha_1} && \ar@{}@<-9pt>[dd]|(.35){}="c" \ar@{}@<-9pt>[dd]|(.65){}="d" & \\ \bE_1 \ar[ru]^-{U_1} \ar[rd]_{U_2} 
&& \bV && \bG \\
& \bH_2 \ar[ru]_{W_2} \ar@/_/[rrru]_{H_2 \cdot R'}="G" &&& \ar@{=>}"c";"d"^(.35)*[@]{\sim}_-{\alpha_2}} \qquad
\cd{ & \bH_1 \ar[rd]^-{W_1} \ar@/^/[rrrd]^-{H_1 \cdot T'}="F" \ar@{}[dd]|(.35){}="a" \ar@{}[dd]|(.65){}="b" \ar@{=>}"a";"b"^(.35)*[@]{\sim}_-{\beta_1} && \ar@{}@<-9pt>[dd]|(.35){}="c" \ar@{}@<-9pt>[dd]|(.65){}="d" & \\ \bE_2 \ar[ru]^-{V_1} \ar[rd]_{V_2} 
&& \bV && \bF \\
& \bH_2 \ar[ru]_{W_2} \ar@/_/[rrru]_{H_2 \cdot T'}="G" &&& \ar@{=>}"c";"d"^(.35)*[@]{\sim}_-{\beta_2}}$$
Suppose also that the diagram
$$\cd{(W_1,H_1) \cdot (1_{\bP},R') \cdot (1_{\bG},T) \ar@{=>}[rr]^-{\alpha \cdot (1_{\bG},T)} \ar@{=>}[d]_{(W_1,H_1) \cdot P_{\Sigma}(\pi)} &&
(W_2,H_2) \cdot (1_{\bP},R') \cdot (1_{\bG},T) \ar@{=>}[d]^{(W_2,H_2) \cdot P_{\Sigma}(\pi)} \\
(W_1,H_1) \cdot (1_{\bP},T') \cdot (1_{\bF},R) \ar@{=>}[rr]_-{\beta \cdot (1_{\bF},R)} && (W_2,H_2) \cdot (1_{\bP},T') \cdot (1_{\bF},R)}$$
commutes in $\mathbf B[\Sigma^{-1}]$. This means that there exist $S_1 \colon \bK \rightarrow \bE_1$, $S_2 \colon \bK \rightarrow \bE_2$ and two 2-isos
$$\cd{ &&& \bH_1 \ar[rd]^-{W_1} \ar@/^/[rrrrrd]^(.35){H_1 \cdot R' \cdot T}="F" \ar@{}[dd]|(.35){}="a" \ar@{}[dd]|(.65){}="b" \ar@{=>}"a";"b"^(.35)*[@]{\sim}_-{\alpha_1} && \ar@{}@<-9pt>[dd]|(.35){}="c" \ar@{}@<-9pt>[dd]|(.65){}="d" &&& \\
\bE_2 \ar@/^1pc/[rrru]^-{V_1}="U1'" \ar@/_1pc/[rrrd]_-{V_2}="U2'" & \bK \ar[r]^-{S_1} \ar[l]_-{S_2} & \bE_1 \ar[ru]^-{U_1} \ar[rd]_{U_2} && \bV &&&& \bB \\
&&& \bH_2 \ar[ru]_{W_2} \ar@/_/[rrrrru]_(.35){H_2 \cdot T' \cdot R}="G" &&&&& \ar@{=>}"c";"d"_(.45)*[@]-{\sim}^-{(\alpha_2 \cdot T) \circ (U_2 \cdot H_2 \cdot \pi)}
\ar@{=>}"U1'";[llllllu]^-{\gamma_1}_(.4)*[@]-{\sim} \ar@{=>}[llllllu];"U2'"^-{\gamma_2}_(.3)*[@]-{\sim}}$$
$\gamma_1 \colon S_2 \cdot V_1 \Rightarrow S_1 \cdot U_1$ and $\gamma_2 \colon S_1 \cdot U_2 \Rightarrow S_2 \cdot V_2$ such that $S_1 \cdot U_1 \cdot W_1 \in \Sigma$ and
the following diagrams commute.
\begin{equation} \label{cond gamma 1}
\vcenter{\cd{S_2 V_1 W_1 \ar@{=>}[r]^-{\gamma_1 \cdot W_1} \ar@{=>}[d]_-{S_2 \cdot \beta_1} & S_1 U_1 W_1 \ar@{=>}[d]^-{S_1 \cdot \alpha_1} \\
S_2 V_2 W_2 & S_1 U_2 W_2 \ar@{=>}[l]^-{\gamma_2 \cdot W_2}}}
\end{equation}
\begin{equation} \label{cond gamma 2}
\vcenter{\cd{ S_1 U_1 H_1 R' T \ar@{=>}[rr]^{S_1 \cdot \alpha_2 \cdot T} && S_1 U_2 H_2 R' T \ar@{=>}[rr]^{S_1U_2H_2 \cdot \pi} &&
S_1 U_2 H_2 T' R \ar@{=>}[d]^-{\gamma_2 \cdot H_2T'R} \\
S_2 V_1 H_1 R' T \ar@{=>}[u]^-{\gamma_1 \cdot H_1R'T} \ar@{=>}[rr]_-{S_2V_1H_1 \cdot \pi} && S_2 V_1 H_1 T' R \ar@{=>}[rr]_-{S_2 \cdot \beta_2 \cdot R} && S_2 V_2 H_2 T' R}}
\end{equation}
Now, we consider the following 2-isos: $$\cd{S_1U_1H_1R' \ar@{=>}[r(.7)]^(.72){S_1 \cdot \alpha_2} & \,\, S_1U_2H_2R'}$$ and
$$\cd{S_1U_1H_1T' \ar@{=>}[rr]^-{\gamma_1^{-1} \cdot H_1 \cdot T'} && S_2V_1H_1T' \ar@{=>}[rr]^-{S_2 \cdot \beta_2} && S_2V_2H_2T' \ar@{=>}[rr]^-{\gamma_2^{-1} \cdot H_2 \cdot T'}
&& S_1U_2H_2T'.}$$
The commutativity of (\ref{cond gamma 2}) and the property BP2 of the bipullback at $\bP$ give us a 2-iso $\delta \colon S_1U_1H_1 \Rightarrow S_1U_2H_2$ such that 
$\delta \cdot R'= S_1 \cdot \alpha_2$ and $$\delta \cdot T'= (\gamma_1^{-1} \cdot H_1 \cdot T') \circ (S_2 \cdot \beta_2) \circ (\gamma_2^{-1} \cdot H_2 \cdot T').$$
So, we can construct a 2-iso $\varphi=[(S_1 \cdot U_1, S_1 \cdot U_2,S_1 \cdot \alpha_1, \delta)] \colon (W_1,H_1) \Rightarrow (W_2,H_2)$.
$$\cd{ & \bH_1 \ar[rd]^-{W_1} \ar@/^/[rrrd]^-{H_1}="F" \ar@{}[dd]|(.35){}="a" \ar@{}[dd]|(.65){}="b" \ar@{=>}"a";"b"^(.35)*[@]{\sim}_-{S_1 \cdot \alpha_1} && \ar@{}@<-9pt>[dd]|(.35){}="c" \ar@{}@<-9pt>[dd]|(.65){}="d" & \\ \bK \ar[ru]^-{S_1 \cdot U_1} \ar[rd]_{S_1 \cdot U_2} && \bV && \bP \\
& \bH_2 \ar[ru]_{W_2} \ar@/_/[rrru]_{H_2}="G" &&& \ar@{=>}"c";"d"^(.35)*[@]{\sim}_-{\delta}}$$
The identities $\varphi \cdot (1_{\bP},R')= \alpha$ and $\varphi \cdot (1_{\bP},T')= \beta$ follow from the diagrams
$$\cd{ &&& \bH_1 \ar[rd]^-{W_1} \ar@/^/[rrrrd]^-{H_1 \cdot R'}="F" \ar@{}[dd]|(.35){}="a" \ar@{}[dd]|(.65){}="b" \ar@{=>}"a";"b"^(.35)*[@]{\sim}_-{\alpha_1} && \ar@{}@<-5pt>[dd]|(.35){}="c" \ar@{}@<-5pt>[dd]|(.65){}="d" && \\
\bK \ar@/^1pc/[rrru]^-{S_1 \cdot U_1}="U1'" \ar@/_1pc/[rrrd]_-{S_1 \cdot U_2}="U2'" & \bK \ar[r]^-{S_1} \ar[l]_-{1_{\bK}} & \bE_1 \ar[ru]^-{U_1} \ar[rd]_{U_2} && \bV &&& \bG \\
&&& \bH_2 \ar[ru]_{W_2} \ar@/_/[rrrru]_-{H_2 \cdot R'}="G" &&&& \ar@{=>}"c";"d"^(.35)*[@]{\sim}_-{\alpha_2}
\ar@{=>}"U1'";[lllllu]^-{1}_(.4)*[@]-{\sim} \ar@{=>}[lllllu];"U2'"^-{1}_(.3)*[@]-{\sim}}$$
and
$$\cd{ &&& \bH_1 \ar[rd]^-{W_1} \ar@/^/[rrrrd]^-{H_1 \cdot T'}="F" \ar@{}[dd]|(.35){}="a" \ar@{}[dd]|(.65){}="b" \ar@{=>}"a";"b"^(.35)*[@]{\sim}_-{S_1 \cdot \alpha_1} && \ar@{}@<2pt>[dd]|(.35){}="c" \ar@{}@<2pt>[dd]|(.65){}="d" && \\
\bE_2 \ar@/^1pc/[rrru]^-{V_1}="U1'" \ar@/_1pc/[rrrd]_-{V_2}="U2'" & \bK \ar[r]^-{1_{\bK}} \ar[l]_-{S_2} & \bK \ar[ru]|-{S_1 \cdot U_1} \ar[rd]|-{S_1 \cdot U_2} && \bV &&& \bF \\
&&& \bH_2 \ar[ru]_{W_2} \ar@/_/[rrrru]_-{H_2 \cdot T'}="G" &&&& \ar@{=>}"c";"d"^(.35)*[@]{\sim}_-{\delta \cdot T'}
\ar@{=>}"U1'";[lllllu]^-{\gamma_1}_(.4)*[@]-{\sim} \ar@{=>}[lllllu];"U2'"^-{\gamma_2}_(.3)*[@]-{\sim}}$$
where the coherence axioms can be deduced from the definition of $\delta$ and the commutativity of (\ref{cond gamma 1}).
It remains to prove the uniqueness of such a 2-cell $\varphi$.
Suppose $\varphi'=[(U_3,U_4,\varepsilon_1, \varepsilon_2)] \colon (W_1,H_1) \Rightarrow (W_2,H_2)$ satisfies $\varphi' \cdot (1_{\bP},R')= \alpha$ and $\varphi' \cdot (1_{\bP},T')= \beta$.
The first identity implies the existence of a diagram
$$\cd{ &&& \bH_1 \ar[rd]^-{W_1} \ar@/^/[rrrrd]^-{H_1 R'}="F" \ar@{}[dd]|(.35){}="a" \ar@{}[dd]|(.65){}="b" \ar@{=>}"a";"b"^(.35)*[@]{\sim}_-{\varepsilon_1} && \ar@{}@<-2pt>[dd]|(.35){}="c" \ar@{}@<-2pt>[dd]|(.65){}="d" && \\
\bE_1 \ar@/^1pc/[rrru]^-{U_1}="U1'" \ar@/_1pc/[rrrd]_-{U_2}="U2'" & \bE_4 \ar[r]^-{U_5} \ar[l]_-{U_6} & \bE_3 \ar[ru]^-{U_3} \ar[rd]_{U_4} && \bV &&& \bG \\
&&& \bH_2 \ar[ru]_{W_2} \ar@/_/[rrrru]_-{H_2 R'}="G" &&&& \ar@{=>}"c";"d"^-{\varepsilon_2 \cdot R'}
\ar@{=>}"U1'";[lllllu]^-{\varepsilon_3}_(.4)*[@]-{\sim} \ar@{=>}[lllllu];"U2'"^-{\varepsilon_4}_(.3)*[@]-{\sim}}$$
where $U_5U_3W_1 \in \Sigma$,
\begin{equation} \label{E4}
U_6 \cdot \alpha_1 = (\varepsilon_3 \cdot W_1) \circ (U_5 \cdot \varepsilon_1) \circ (\varepsilon_4 \cdot W_2)
\end{equation}
and
\begin{equation} \label{E4'}
U_6 \cdot \alpha_2 = (\varepsilon_3 \cdot H_1R') \circ (U_5 \cdot \varepsilon_2 \cdot R') \circ (\varepsilon_4 \cdot H_2R').
\end{equation}
Since $\varphi'=[(U_5U_3,U_5U_4,U_5 \cdot \varepsilon _1, U_5 \cdot \varepsilon_2)]$, the second identity means that there exists a diagram
$$\cd{ &&& \bH_1 \ar[rd]^-{W_1} \ar@/^/[rrrrd]^-{H_1 T'}="F" \ar@{}[dd]|(.35){}="a" \ar@{}[dd]|(.65){}="b" \ar@{=>}"a";"b"^(.35)*[@]{\sim}_-{U_5 \cdot \varepsilon_1} && \ar@{}@<-9pt>[dd]|(.35){}="c" \ar@{}@<-9pt>[dd]|(.65){}="d" && \\
\bE_2 \ar@/^1pc/[rrru]^-{V_1}="U1'" \ar@/_1pc/[rrrd]_-{V_2}="U2'" & \bE_5 \ar[r]^-{U_7} \ar[l]_-{U_8} & \bE_4 \ar[ru]|-{U_5U_3} \ar[rd]|-{U_5U_4} && \bV &&& \bF \\
&&& \bH_2 \ar[ru]_{W_2} \ar@/_/[rrrru]_-{H_2 T'}="G" &&&& \ar@{=>}"c";"d"^-{U_5 \cdot \varepsilon_2 \cdot T'}
\ar@{=>}"U1'";[lllllu]^-{\varepsilon_5}_(.4)*[@]-{\sim} \ar@{=>}[lllllu];"U2'"^-{\varepsilon_6}_(.3)*[@]-{\sim}}$$
where $U_7U_5U_3W_1 \in \Sigma$,
\begin{equation} \label{E5}
U_8 \cdot \beta_1 = (\varepsilon_5 \cdot W_1) \circ (U_7U_5 \cdot \varepsilon_1) \circ (\varepsilon_6 \cdot W_2)
\end{equation}
and
\begin{equation} \label{E5'}
U_8 \cdot \beta_2 = (\varepsilon_5 \cdot H_1T') \circ (U_7U_5 \cdot \varepsilon_2 \cdot T') \circ (\varepsilon_6 \cdot H_2T').
\end{equation}
Now, since $S_2V_1W_1$ and $V_1W_1$ are both in $\Sigma$, by axioms BF2, BF3 and BF4, there exist two arrows $U_9$ and $U_{10}$ with $U_9 \in \Sigma$ and a 2-iso
$\varepsilon_7 \colon U_9U_8 \Rightarrow U_{10}S_2$.
$$\cd{\bE_6 \ar[d]_-{U_9} \ar[r]^-{U_{10}} & \bK \ar[d]^-{S_2} \\
\bE_5 \ar[r]_-{U_8} \ar@{}[ru]|(.35){}="a" \ar@{}[ru]|(.65){}="z" \ar@{=>}"a";"z"_{\varepsilon_7} \ar@{}@<-6pt>"a";"z"^*[@]{\sim} & \bE_2}$$
Let us consider the 2-iso
$$\cd{U_9U_7U_6U_1 \ar@{=>}[r]^-{U_9U_7 \cdot \varepsilon_3} & U_9U_7U_5U_3 \ar@{=>}[r]^-{U_9 \cdot \varepsilon_5^{-1}} & U_9U_8V_1 \ar@{=>}[r]^-{\varepsilon_7 \cdot V_1}
& U_{10}S_2V_1 \ar@{=>}[r]^-{U_{10} \cdot \gamma_1} & U_{10}S_1U_1.}$$
Then, since $U_1W_1$ and $W_1$ are both in $\Sigma$, using the axiom BF4 twice, we get a 1-cell $U_{11} \colon \bE_7 \rightarrow \bE_6$ in $\Sigma$ and
a 2-iso $\varepsilon_8 \colon U_{11}U_9U_7U_6 \Rightarrow U_{11}U_{10}S_1$ such that
\begin{equation}\label{E7}
\varepsilon_8 \cdot U_1 =(U_{11}U_9U_7 \cdot \varepsilon_3) \circ (U_{11}U_9 \cdot \varepsilon_5^{-1}) \circ (U_{11} \cdot \varepsilon_7 \cdot V_1) \circ (U_{11}U_{10} \cdot \gamma_1).
\end{equation}
Using the identities (\ref{cond gamma 1}), (\ref{E4}), (\ref{E5}) and (\ref{E7}), we get that
$$(U_{11}U_9 \cdot \varepsilon_6 \cdot W_2) \circ (U_{11} \cdot \varepsilon_7 \cdot V_2W_2) =
(U_{11}U_9U_7 \cdot \varepsilon_4 \cdot W_2) \circ (\varepsilon_8 \cdot U_2W_2) \circ (U_{11}U_{10} \cdot \gamma_2 \cdot W_2).$$
Therefore, since $W_2 \in \Sigma$, by the axiom BF4, there exists an arrow $U_{12} \colon \bE_8 \rightarrow \bE_7$ in $\Sigma$ such that
\begin{equation} \label{E8}
(U_{12}U_{11}U_9 \cdot \varepsilon_6) \circ (U_{12}U_{11} \cdot \varepsilon_7 \cdot V_2) =
(U_{12}U_{11}U_9U_7 \cdot \varepsilon_4) \circ (U_{12} \cdot \varepsilon_8 \cdot U_2) \circ (U_{12}U_{11}U_{10} \cdot \gamma_2).
\end{equation}
Finally, to prove that $\varphi'=\varphi$, we consider the following diagram
$$\cd{ &&&&& \bH_1 \ar[rd]^-{W_1} \ar@/^/[rrrrd]^-{H_1}="F" \ar@{}[dd]|(.35){}="a" \ar@{}[dd]|(.65){}="b" \ar@{=>}"a";"b"^(.35)*[@]{\sim}_-{S_1 \cdot \alpha_1} && \ar@{}@<-9pt>[dd]|(.35){}="c" \ar@{}@<-9pt>[dd]|(.65){}="d" && \\
\bE_3 \ar@/^1.5pc/@<5pt>[rrrrru]^-{U_3}="U1'" \ar@/_1.5pc/@<-5pt>[rrrrrd]_-{U_4}="U2'" && \bE_8 \ar[rr]^-{U_{12}U_{11}U_{10}} \ar[ll]_-{U_{12}U_{11}U_9U_7U_5}
&& \bK \ar[ru]^-{S_1U_1} \ar[rd]_{S_1U_2} && \bV &&& \bP \\
&&&&& \bH_2 \ar[ru]_{W_2} \ar@/_/[rrrru]_-{H_2}="G" &&&& \ar@{=>}"c";"d"^(.35)*[@]{\sim}_-{\delta}}$$
with the 2-isos
$$(U_{12}U_{11}U_9 \cdot \varepsilon_5^{-1}) \circ (U_{12}U_{11} \cdot \varepsilon_7 \cdot V_1) \circ (U_{12}U_{11}U_{10} \cdot \gamma_1) \colon
U_{12}U_{11}U_9U_7U_5U_3 \Rightarrow U_{12}U_{11}U_{10}S_1U_1$$
and
$$(U_{12}U_{11}U_{10} \cdot \gamma_2) \circ (U_{12}U_{11} \cdot \varepsilon_7^{-1} \cdot V_2) \circ (U_{12}U_{11}U_9 \cdot \varepsilon_6^{-1}) \colon
U_{12}U_{11}U_{10}S_1U_2 \Rightarrow U_{12}U_{11}U_9U_7U_5U_4.$$
To prove the first coherence axiom, we use identities (\ref{cond gamma 1}) and (\ref{E5}), while for the second one, we use the universal property of the bipullback.
Indeed, to prove it, it suffices to compose each 2-cell with both $T'$ and $R'$. The one composed with $T'$ can be deduced from the definition of $\delta$ and (\ref{E5'}),
while the one composed with $R'$ follows from the definition of~$\delta$, (\ref{E4'}), (\ref{E7}) and (\ref{E8}).
\end{proof}

\section{The snail lemma for fractors}

In this section, $\cA$ is a pointed exact category. 

\begin{Text}\label{TextIntGrpd}{\rm

The 2-category $\Grpd(\cA)$ of internal groupoids in $\cA$ has bipullbacks and, in particular, bikernels (in fact, $\Grpd(\cA)$ has strong h-pullbacks, and so it has bipullbacks, 
see~\cite{FractEV, JMMVFibr}). Moreover, there are two 2-functors
$$\pi_0 \colon \Grpd(\cA) \to \cA \quad \text{and} \quad \pi_1 \colon \Grpd(\cA) \to \Grp(\cA)$$
(where $\cA$ and $\Grp(\cA)$, the category of internal groups in $\cA$, are seen as discrete 2-categories)
respectively defined by the following coequalizer and kernel
$$\xymatrix{A_1 \ar@<0,5ex>[r]^-{d} \ar@<-0,5ex>[r]_-{c} & A_0 \ar[r]^-{\eta_{\bA}} & \pi_0(\bA)}
\;\;\;\;\;\;\;\;\;\;
\xymatrix{\pi_1(\bA) \ar[r]^-{\epsilon_{\bA}} & A_1 \ar[r]^-{\langle d,c \rangle} & A_0 \times A_0}$$
where $d$ and $c$ are respectively the domain and codomain maps of the internal groupoid~$\bA$.
For an internal functor $F \colon \bA \to \bB$, its bikernel is defined as the bipullback
$$\xymatrix{\bK(F) \ar[r]^-{K(F)} \ar[d] & \bA \ar[d]^-{F} \\
0 \ar@{}[ru]|(.35){}="a" \ar@{}[ru]|(.65){}="z" \ar@{=>}"a";"z"_-{k(F)} \ar@{}@<-6pt>"a";"z"^*[@]{\sim} \ar[r] & \bB}$$
and it is represented by
$$\xymatrix{ \bK(F) \ar[r]^-{K(F)} & \bA \ar[r]^-{F} & \bB.}$$
The main result in~\cite{SnailMMVShort} states that for any internal functor $F$, there is an exact sequence
$$\xymatrix{\pi_1(\bK(F)) \ar[rr]^-{\pi_1(K(F))} & &  \pi_1(\bA) \ar[r]^{\pi_1(F)} & \pi_1(\bB) \ar[r]^-{ } 
& \pi_0(\bK(F)) \ar[rr]^-{\pi_0(K(F))} & & \pi_0(\bA) \ar[r]^{\pi_0(F)} & \pi_0(\bB).}$$
Here, the exactness at $B$ of 
$$\xymatrix{A \ar[r]^-{f} & B \ar[r]^-{g} & C}$$
is intended in the sense that $f$ factors as a regular epimorphism followed by the kernel of~$g$.
}\end{Text}

\begin{Text}\label{TextFract}{\rm
The class of weak equivalences in $\Grpd(\cA)$ has a right calculus of fractions (in the bicategorical sense)~\cite{FractEV}.
The bicategory of fractions of $\Grpd(\cA)$ with respect to this class of weak equivalences has been described in~\cite{MMV} 
(see also~\cite{AMMV} if $\cA$ is semi-abelian, \cite{JV} if $\cA$ is monadic, \cite{AbbV} if $\cA$ has enough regular 
projective objects, and~\cite{DR} for a description in terms of anafunctors). The objects are internal groupoids, and the 
arrows are particular profunctors called fractors: a fractor $E \colon \bA \looparrowright \bB$  is a diagram of the form
$$\xymatrix{ & R \ar[ld]_{\overline \sigma} \ar@<-0,5ex>[rd]_{d} \ar@<0,5ex>[rd]^{c} 
& & R[{\sigma}] \ar@<-0,5ex>[ld]_{\sigma_1} \ar@<0,5ex>[ld]^{\sigma_2} \ar[rd]^{\overline \rho} \\
A_1 \ar@<-0,5ex>[rd]_{d} \ar@<0,5ex>[rd]^{c} & & E \ar@{->>}[ld]^{\sigma} \ar[rd]_{\rho} 
& & B_1 \ar@<-0,5ex>[ld]_{d} \ar@<0,5ex>[ld]^{c} \\
& A_0 & & B_0}$$
where
\begin{enumerate}
\item[-] $\sigma$ is a regular epimorphism, and $R[{\sigma}]$ is its kernel pair;
\item[-] $\rho$ coequalizes $d, c \colon R \rightrightarrows E$;
\item[-] $(\overline{\sigma},\sigma)$ and $(\overline{\rho},\rho)$
are discrete fibrations. 
\end{enumerate}
Given two fractors $E \colon \bA \looparrowright \bB$ and 
$E' \colon \bA \looparrowright \bB$, a 2-cell is just an arrow $E \to E'$ in $\cA$
satisfying suitable compatibility conditions.\\
The main result in~\cite{MMV} states that the bicategory of fractions of $\Grpd(\cA)$ with respect to weak equivalences is the embedding
$$\cF \colon \Grpd(\cA) \to \Fract(\cA)$$
of functors into fractors. Therefore, using Proposition~\ref{PropBPBFract1}, we have the following result.
}\end{Text}

\begin{Proposition}\label{PropBipbFract}
\hfill
\begin{enumerate}
\item The bicategory of fractors \/ $\Fract(\cA)$ has bipullbacks and moreover the 2-functor $\cF \colon \Grpd(\cA) \to \Fract(\cA)$ 
preserves bipullbacks.
\item In particular, $\Fract(\cA)$ has bikernels and $\cF \colon \Grpd(\cA) \to \Fract(\cA)$ preserves bikernels.
\end{enumerate}
\end{Proposition}

From Lemma~4.5 in~\cite{SnailMMVShort}, we also know that the 2-functors $\pi_0$ and $\pi_1$ convert weak
equivalences in isomorphisms, so that thay can be extended to the bicategory of fractions.
$$\xymatrix{\Grpd(\cA) \ar[r]^-{\cF} \ar[rd]_{\pi_0} & \Fract(\cA) \ar[d]^{\pi_0} \\
\ar@{}[ru]|>>>>{\cong} & \cA} \;\;\;\;\;\;\;\;
\xymatrix{\Grpd(\cA) \ar[r]^-{\cF} \ar[rd]_{\pi_1} & \Fract(\cA) \ar[d]^{\pi_1} \\
\ar@{}[ru]|>>>>{\cong} & \Grp(\cA)}$$
Now we are ready to extend the exact sequence of~\ref{TextIntGrpd} to fractors.

\begin{Theorem}\label{ThFractSnail}
Let $F \colon \bA \looparrowright \bB$ be a fractor between groupoids in $\cA$, together with its bikernel 
$K(F) \colon \bK(F) \to \bA$. There exists an exact sequence
$$\xymatrix{\pi_1(\bK(F)) \ar[rr]^-{\pi_1(K(F))} & &  \pi_1(\bA) \ar[r]^{\pi_1(F)} & \pi_1(\bB) \ar[r]^-{ } 
& \pi_0(\bK(F)) \ar[rr]^-{\pi_0(K(F))} & & \pi_0(\bA) \ar[r]^{\pi_0(F)} & \pi_0(\bB).}$$
\end{Theorem}

\begin{proof}
Since the class of weak equivalences in $\Grpd(\cA)$ has a right calculus of fractions (in the bicategorical sense)~\cite{FractEV}, by construction of 
$\Grpd(\cA)[\Sigma^{-1}] \simeq \Fract(\cA)$,
the fractor $F \colon \bA \looparrowright \bB$ has a tabulation, i.e., there exist two internal functors 
$$\xymatrix{\bA & \ar[l]_-{S} \bF \ar[r]^-{R} & \bB}$$ 
with $S$ a weak equivalence and such that
$\cF(S) \cdot F \cong \cF(R)$. We also consider the bikernel
$$\xymatrix{\bK(R) \ar[r]^-{K(R)} & \bF \ar[r]^-{R} & \bB}$$
of $R$ in $\Grpd(\cA)$. By Proposition~\ref{PropBipbFract},
$$\xymatrix{\bK(R) \ar[rr]^-{\cF(K(R))} & & \bF \ar[r]^-{\cF(R)} & \bB}$$
is a bikernel in $\Fract(\cA)$. Therefore, since $\cF(S)$ is an equivalence, also
$$\xymatrix{\bK(R) \ar[rr]^-{\cF(K(R)) \cdot \cF(S)} & & \bA \ar[r]^-{F} & \bB}$$
is a bikernel in $\Fract(\cA)$. In other words, the comparison $S'$ in the following diagram is an equivalence.
$$\xymatrix{\bK(R) \ar[rr]^-{\cF(K(R))} \ar[d]_{S'} & & \bF \ar[rr]^-{\cF(R)} \ar[d]_{\cF(S)} & & \bB \ar[d]^{1_{\bB}} \\
\bK(F) \ar[rr]_-{K(F)} \ar@{}[rru]|{\cong} & & \bA \ar[rr]_-{F} & \ar@{}[u]|{\cong} & \bB}$$
Now we can construct an exact sequence starting from the functor $R \colon \bF \to \bB$ as in~\ref{TextIntGrpd}, 
and we get the first line of the following diagram.
$$\resizebox{\textwidth}{!}{\xymatrix{\pi_1(\bK(R)) \ar[rr]^-{\pi_1(K(R))} \ar[d]_{\cong} & & \pi_1(\bF) \ar[rr]^-{\pi_1(R)} \ar[d]_{\cong} & & \pi_1(\bB) \ar[r] \ar[d]_{\cong} 
& \pi_0(\bK(R)) \ar[rr]^-{\pi_0(K(R))} \ar[d]^{\cong} & & \pi_0(\bF) \ar[rr]^-{\pi_0(R)} \ar[d]^{\cong} & & \pi_0(\bB) \ar[d]^{\cong} \\
\pi_1(\bK(R)) \ar[rr]^-{\pi_1(\cF(K(R)))} \ar[d]_{\pi_1(S')} & & \pi_1(\bF) \ar[rr]^-{\pi_1(\cF(R))} \ar[d]_{\pi_1(\cF(S))} & & \pi_1(\bB) \ar[d]_{\id} & 
\pi_0(\bK(R)) \ar[rr]^-{\pi_0(\cF(K(R)))} \ar[d]^{\pi_0(S')} & & \pi_0(\bF) \ar[rr]^-{\pi_0(\cF(R))} \ar[d]^{\pi_0(\cF(S))} & & \pi_0(\bB) \ar[d]^{\id} \\
\pi_1(\bK(F)) \ar[rr]_-{\pi_1(K(F))} & & \pi_1(\bA) \ar[rr]_-{\pi_1(F)} & & \pi_1(\bB) & 
\pi_0(\bK(F)) \ar[rr]_-{\pi_0(K(F))} & & \pi_0(\bA) \ar[rr]_-{\pi_0(F)} & & \pi_0(\bB)}}$$
Since $S'$ and $\cF(S)$ are equivalences, all the columns are isomorphisms and the proof is complete.
\end{proof}

\begin{Text}\label{TextProper}{\rm
In~\cite{SnailMMVShort}, the $\pi_0$-$\pi_1$-exact sequence associated with an internal functor $F \colon \bA \to \bB$ is obtained under the 
assumption that the groupoids $\bA, \bB$ and $\bK(F)$ are proper. Here we can omit this assumption, because our base category $\cA$ is 
exact, so that all groupoids are proper.
}\end{Text}

\end{document}